\RequirePackage{fix-cm}
\documentclass[smallextended]{svjour3}       % onecolumn (second format)
\smartqed  % flush right qed marks, e.g. at end of proof
\usepackage{graphicx}
\usepackage{amsmath,amscd,amssymb}
\begin{document}

\title{On a Class of Gradient Almost Ricci Solitons%\thanks{Grants or other notes
%about the article that should go on the front page should be
%placed here. General acknowledgments should be placed at the end of the article.}
}
%\subtitle{Do you have a subtitle?\\ If so, write it here}

%\titlerunning{Short form of title}        % if too long for running head

\author{Sinem G\"{u}ler        %etc.
}

%\authorrunning{Short form of author list} % if too long for running head

\institute{S.  G\"{u}ler \at
              Istanbul Sabahattin Zaim University, Department of Industrial Engineering, Halkal\i, Istanbul, Turkey \\
              \email{sinem.guler@izu.edu.tr}           %  \\
%             \emph{Present address:} of F. Author  %  if needed
}

\date{Received: date / Accepted: date}
% The correct dates will be entered by the editor

\maketitle

\begin{abstract}
In this study, we provide some classifications for  half-conformally flat gradient $f$-almost Ricci solitons, denoted by $(M, g, f)$,  in both Lorentzian and neutral signature. First, we prove that if $||\nabla f||$ is a non-zero constant,  then $(M, g, f)$ is locally isometric to a {warped product} of the form $I \times_{\varphi} N$, where $I \subset \mathbb{R}$ and $N$ is of constant sectional curvature. On the other hand, if $||\nabla f|| = 0$,  then it is locally a {Walker manifold}. Then, we construct an example of  4-dimensional steady gradient $f$-almost Ricci solitons in neutral signature. At the end,  we  give more physical applications of gradient Ricci solitons endowed with the standard static spacetime metric.
% \PACS{PACS code1 \and PACS code2 \and more}
\end{abstract}

\section{Introduction}
\label{intro}

Hamilton  introduced the concept of the Ricci flow to prove the Poincare Conjecture in the late of 20th century \cite{hamilton}.  Poincare Conjecture was one of the very deep unsolved problem which aims to classify  all compact three dimensional manifolds.  In the  1900's, Poincare asked if a simply-connected closed three  manifold is necessarily the three sphere $\mathbb{S}^3$. For this purpose, Hamilton introduced the Ricci flow as a partial differential equation 
$\frac{\partial g(t)}{\partial t}=-2Ric(g(t))$, which evolves the metric in a Riemannian manifold to make it rounder. By choosing harmonic coordinates, it can be seen that the Ricci flow is a heat-type equation and the characteristic property of such equations is the maximum principle, which guarantees that this rounding metric happens in some specific case. Hence, as in the expectation of Hamilton, after the evoluation of the metric, the manifold will become of constant curvature. 
 Moreover, Hamilton proved that for any smooth metric $g_0$ on a compact Riemannian manifold $M^n$, there exists a unique solution $g(t)$ of {Ricci flow}  defined on some interval $[0, \varepsilon)$, for some $\varepsilon > 0$, with the initial condition $g(0) = g_0$, \cite{hamilton}. Even in the non-compact case, a complete solution of {Ricci flow} exists when the sectional curvatures of $g_0$ are bounded, \cite{shi}.

Let $(M^n, g(t))$ be a solution of the Ricci flow and suppose that $\varphi_t : M^n \rightarrow M^n$ is a time-dependent family of diffeomorphisms satisfying $\varphi_0 = Id$ and $\sigma(t)$ is a time-dependent scale factor satisfying  $\sigma(0) = 1$. If  \
$g(t) = \sigma(t)\varphi_t^*g(0) $
holds,  then the solution $(M^n, g(t))$ is called a Ricci soliton.
Thus, one can regard the Ricci solitons as the  fixed points of the Ricci flow, which change only by a  diffeomorphism and a rescaling. If we take the derivative of the last relation at $t=0$, by using the definition of the Lie derivative, the {Ricci flow} equation and the initial conditions for $\varphi_t$ and $\sigma(t)$, we obtain
\begin{equation}
\label{Ricci soliton}
Ric+\frac{1}{2}\mathcal{L}_{V}g=\lambda g,
\end{equation}
where $V=\frac{d\varphi_t}{dt}$, $\mathcal{L}_{V}g$ is the Lie derivative of the metric $g$ in the direction of $V$ and $\sigma^{'}(0)=2\lambda$, for some real constant $\lambda$. As a result, the triple $(M,g,V)$ is said to be a Ricci soliton, where $V$ is the potential field and $\lambda$ is the soliton constant.   Also, there is a terminology according to the sign of $\lambda$: a Ricci soliton is {steady}  if $\lambda=0$, and {expanding} or {shrinking} if $\lambda >0$ or $\lambda <0$, respectively. If the potential vector field $X$ is gradient, that is $X=\nabla f$, for some smooth function $f$, then the equation $\eqref{Ricci soliton}$ reduces to the form
\begin{equation}
\label{gradient Ricci soliton}
Ric+Hessf=\lambda g.
\end{equation}
Thus the triple $(M,g,f)$ satisfying $\eqref{gradient Ricci soliton}$ is called a gradient Ricci soliton and $f$ is called a potential field. Mostly, in a gradient Ricci soliton $\lambda$ is a real constant. But, we are also interested in the general case in which $\lambda$ is a smooth function on $M$. In that case, $(M,g,f)$ is called a gradient almost Ricci soliton.
In these regards, Ricci solitons and gradient Ricci solitons are natural generalizations of Einstein manifolds. In literature, there are many other different generalizations of Einstein manifolds and gradient Ricci solitons, such as quasi Einstein manifolds \cite{case1,case2,bejguler}, $(m,\rho)$-quasi Einstein manifolds \cite{guler,huang,ghosh}, generalized quasi Einstein manifolds \cite{catino,guler.1,ahmad} and etc. 

%The most general definition is as follows:

In this paper, we will deal with another generalization of Ricci solitons, given as follows:

An {$h$-almost Ricci soliton} is a complete Riemannian manifold $(M^n, g)$ with
a vector field $X \in \chi(M)$ and two smooth real valued  functions $\lambda $ and $h$ satisfying the equation \cite{gomes}
\begin{equation}
\label{h-almost}
Ric +\frac{h}{2}\mathcal{L}_Xg=\lambda g. 
\end{equation}
For the sake of convenience, we denote an $h$-almost Ricci soliton by $(M^n, g,X, h, \lambda )$.   If $\lambda$ is constant, it is called an {$h$-Ricci soliton}.  When $X=\nabla f$  for some smooth function $f$, we call $(M^n, g, \nabla f, h, \lambda )$ a {gradient $h$-almost Ricci soliton} with potential function $f$. In this case, the fundamental equation $\eqref{h-almost}$ can be rewritten as \cite{yun}:
\begin{equation}
\label{gradient h-almost}
Ric +h.Hess f=\lambda g
\end{equation}
that  constitutes the main theme of this work. This kind of manifolds are closely related to the warped products so we should mention about the warped products, \cite{oneill}: 
For two pseudo-Riemannian manifolds  $(B,g_B)$ and $(F,g_F)$, the warped product $B\times_{\varphi} F$ with respect to warping function  $\varphi \in \mathcal{C}_{>0}^{\infty}(B)$ is defined as the product manifold  $B\times F$ endowed with the metric $g=g_B+\varphi^2 g_F$. After some straightforward calculations, all components of the Ricci tensor of $(B\times_{\varphi}F, g)$ can be found as follows \cite{oneill}
\begin{itemize}
\item[(1)] ${Ric}(X,Y)=Ric_B(X,Y)-\frac{m}{\varphi}Hess{\varphi}(X,Y)$,
\item[(2)] ${Ric}(X,V)=0$,
\item[(3)] ${Ric}(V,W)=Ric_F(V,W)-\Big[\frac{\Delta \varphi}{\varphi}+(m-1)\frac{||grad \varphi||^2}{\varphi^2}\Big]{g}(V,W)$,
\end{itemize}
for all horizontal vectors $X, Y$ and vertical vectors $V,W$ where $Hess \varphi =\nabla d\varphi$ denotes the Hessian of
a smooth function $\varphi$ on  $(B, g_B)$, $m > 1$ is the dimension of  $(F, g_F)$.
Einstein and quasi Einstein manifolds are closely related to the warped products. For example, 
% if $(M^n,g)$ is a quasi Einstein with $\alpha =-\frac{1}{n-2}$, then it is conformally Einstein, i.e., the conformal metric $\tilde{g}=e^{-\frac{2}{n-2}f}g$ is Einstein. On the other hand, 
if $B\times_{\varphi}F$ is an Einstein warped product, then $(B,g,f,\alpha)$ is quasi Einstein, with $f=-dimF (log \varphi)$ and $\alpha=\frac{1}{dimF}$. Moreover, if we have a suitable fiber, the converse of this statement is also true. Thus, this is a natural way to construct a quasi Einstein manifolds and warped product metrics.

Additionally, if  $B\times_{\varphi} F$ satisfies $Ric(X,Y)=\lambda g(X,Y)$ for all $X,Y \in \chi(B)$, then the base manifold $(B,g_B,\nabla \varphi, -\frac{m}{\varphi}, \lambda)$ is the gradient  $(-\frac{m}{\varphi})$ almost Ricci soliton. Conversely, if fiber $(F, g_F)$ is Einstein with $Ric_F=\mu g_F$, then the necessary and sufficient condition for 
 $B\times_{\varphi} F$ to be an Einstein with $Ric =\lambda g$  is that the base $(B, g_B)$ is a gradient  $(-\frac{m}{\varphi})$-almost Ricci soliton with potential function $\varphi$ and soliton function $\lambda$ satisfying
$\mu = \varphi \Delta \varphi + (m- 1)||\nabla \varphi||^2 + \lambda \varphi^2$.

Motivated by these results, in this study first we analyse the half conformally flat (i.e. self-dual or anti-self-dual) four dimensional gradient $f$-almost Ricci solitons (that is, throughout this study the function $h$ is identified  by the potential function $f$). Indeed, we restrict ourselves to the particular case in which $h$ equals to the potential function $f$. In \cite{garcia}, half conformally flat gradient Ricci almost solitons are investigated, showing that they are locally conformally flat in a neighbourhood of any point where the gradient of the potential function is
non-null. In opposition, if the gradient of the potential function is null, then the soliton is a Walker manifold. The first case corresponds to non-degenerate level hypersurfaces, whereas the second case corresponds to degenerate level hypersurfaces and gives rise to the  isotropic solitons, \cite{garcia}. From the inspiration of these results, we will extend this problem to the gradient $f$-almost Ricci solitons. In the last part,  we also give more physical applications of gradient Ricci solitons by using  the standard static spacetime metric.  We provide the characterizations of certain manifolds satisfying Ricci-Hessian class type equations endowed with standard static spacetime metric, with respect to their
fundamental equations.

\section{Preliminaries}

Let $(M, g)$ be a pseudo-Riemannian manifold with Levi-Civita connection $\nabla$. The  curvature operator is defined by $R(X,Y)=[\nabla_X,\nabla_Y]-\nabla_{[X,Y]}$, for any $X,Y\in \chi(M)$. Then the Ricci tensor $Ric$ and the scalar curvature $r$ are defined by $Ric(X,Y)=trace\{Z \rightarrow R(X,Y)Z\}$ and $r =trace\{Ric\}$ respectively, where
$Q$ denotes the Ricci operator defined by $g(QX, Y)=Ric(X,Y)$. As we mentioned before, $Hessf$ denotes the Hessian tensor defined by $Hessf (X,Y)=(\nabla_Xdf )(Y)=XY(f ) - (\nabla_XY)(f )$.

\subsection{Some key lemmas about gradient $f$-almost Ricci solitons: }
First we give the following  which will help us to characterize the gradient $f$-almost Ricci solitons. 

\begin{lemma} \label{lemma:1}
Let $(M^n, g, \nabla f, f, \lambda )$ be a gradient  $f$-almost Ricci soliton. Then the following identities hold:
\begin{itemize}
\item[(1)] $r+f\Delta f=\lambda n$
\item[(2)] $\nabla r=-2 (\nabla f)\Delta f+2fRic(\nabla f)+2hes_f(\nabla f)+2(n-1)\nabla \lambda$
\item[(3)] $fR(X,Y,Z,\nabla f)=d\lambda(X)g(Y,Z)-d\lambda (Y)g(X,Z)\\ - [df(X)Hessf(Y,Z)-df(Y)Hessf(X,Z)]-[(\nabla_XRic)(Y,Z)-(\nabla_YRic)(X,Z)]$
\end{itemize}
for all vector fields $X,Y,Z$ on $M$. 
\end{lemma}

\begin{proof}
Contracting the fundamental equation of $(M^n, g, \nabla f, f, \lambda )$, the first assertion is directly obtained. If we combine (1) with the Ricci identity, the second assertion can be verified. Thus, (1)-(2) and the Ricci identity yield the last relation.

\end{proof}

The Weyl tensor $\mathcal{W}$ and the Cotton tensor $C$ are defined as follows:
\begin{align}
\label{weyl}
\mathcal{W}(X,Y)Z=&R(X,Y)Z+\frac{r}{(n-1)(n-2)}[g(X,Z)Y-g(Y,Z)X]\\ \notag
+&\frac{1}{n-2}[Ric(Y,Z)X-Ric(X,Z)Y+g(Y,Z)QX-g(X,Z)QY],
\end{align}
\begin{align}
\label{cotton}
C(X,Y,Z)=&(\nabla_XRic)(Y,Z)-(\nabla_YRic)(X,Z)\\ \notag
-&\frac{1}{2(n-1)}[g(Y,Z)dr(X)-g(X,Z)dr(Y)],
\end{align}
for all vector fields $X,Y,Z$ on $M$. Then by using the equations $\eqref{weyl}$ and $\eqref{cotton}$ and Lemma $\ref{lemma:1}$, we can immediately obtain the following result that will be given without proof:

\begin{lemma}\label{lemma:2}
Let $(M^n, g, \nabla f, f, \lambda )$ be a gradient  $f$-almost Ricci soliton. Then
\begin{align}\footnotesize 
\label{lemma 2}
\mathcal{W}(X,Y,Z,\nabla f)=&\frac{1}{(n-1)(n-2)}[Ric(X,\nabla f)g(Y,Z)-Ric(Y,\nabla f)g(X,Z)]\\ \notag
-&\frac{r}{(n-1)(n-2)}[g(X,\nabla f)g(Y,Z)-g(Y,\nabla f)g(X,Z)]\\ \notag
+&\frac{1}{(n-2)}[g(X,\nabla f)Ric(Y,Z)-g(Y,\nabla f)Ric(X,Z)]\\ \notag
-&\frac{1}{f}[df(X)Hessf(Y,Z)-df(Y)Hessf(X,Z)] \\ \notag 
+&\frac{\Delta f}{f(n-1)}[df(X)g(Y,Z)-df(Y)g(X,Z)] \\ \notag 
-&\frac{1}{2f(n-1)}[X(||\nabla f||^2)g(Y,Z)-Y(||\nabla f||^2)g(X,Z)]-\frac{1}{f}C(X,Y,Z)
\end{align}
for all vector fields $X,Y,Z$ on $M$.
\end{lemma}

\subsection{Half-conformal flatness of 4-dimesional manifolds}

Let $(V, <.,.>)$ be an inner product vector space and let $<<.,.>>$ be the induced inner product on the space of two forms $\Lambda^2(V)$. For  a given orientation $vol_V$, the Hodge star operator $*:\Lambda^2(V)\rightarrow \Lambda(V)$ given by $\alpha \wedge *\beta=<<\alpha,\beta>>vol_V$ satisfies $*^2=Id$ and induces a decomposition $\Lambda^2=\Lambda_{+}^2 \oplus \Lambda_{-}^2$, where $\Lambda_{+}^2=\{\alpha \in \Lambda^2: *\alpha=\alpha\}$ and $\Lambda_{-}^2=\{\alpha \in \Lambda^2: *\alpha=-\alpha\}$. $\Lambda_{+}^2$ denotes the space of self-dual and $\Lambda_{-}^2$ denotes the space of anti-self-dual two forms. Let $W:\Lambda^2(V)\rightarrow \Lambda(V)$ be the corresponding endomorphism associated to the Weyl conformal tensor. Then $W$ can be decomposed under the action of $SO(V,<.,.>)$ as $W=W^{+} \oplus W^-$, where $W^{+}=\frac{W+*W}{2}$ is the self-dual and   $W^{-}=\frac{W-*W}{2}$ is the anti-self-dual Weyl conformal curvature tensor. 
Half conformally flat metrics are  known as self-dual or anti-self dual if
$W^{-} = 0$ or $W^{+} = 0$, respectively.

We will use the following characterization of self-dual algebraic curvature tensors in sequel:

\begin{lemma}\label{lemma:brozos}\cite{garcia}
Let $(V, <.,.>)$ be an oriented four-dimensional inner product space of neutral signature. Then the followings hold:
\begin{enumerate}
\item[(1)] An algebraic curvature tensor $R$ is self-dual if and only if for any positively oriented orthonormal basis $\{e_1,e_2,e_3,e_4\}$ ,
\begin{equation}
\label{lemma brozos 1}
W(e_1,e_i,X,Y)=\sigma_{ijk}\varepsilon_j\varepsilon_kW(e_j,e_k,X,Y), \ \ \ \forall X,Y \in V
\end{equation}
for $i,j,k\in \{2,3,4\}$, where $\sigma_{ijk}$ denotes the signature of the corresponding permutation. 
\item[(2)] An algebraic curvature tensor $R$ is self-dual if and only if for any positively oriented pseudo-orthonormal basis $\{T,U,V,W\}$   (i.e., the non-zero inner products are $<T,V>=<U,W>=1$) and for every $X,Y\in V$,
\begin{equation}
\label{lemma brozos 2}
W(T,V,X,Y)=W(U,W,X,Y), \ \ W(T,W,X,Y)=0, \ \ W(U,V,X,Y)=0
\end{equation}
for $i,j,k\in \{2,3,4\}$, where $\sigma_{ijk}$ denotes the signature of the corresponding permutation. 
\end{enumerate}
\end{lemma}

\section{Half-conformally flat gradient $f$-almost Ricci solitons}

The aim of this section is to analyze four-dimensional half conformally flat (i.e. self-dual or anti-self-dual) gradient $f$-almost Ricci solitons. Since we work at the local sets, let $p\in M$ and orient $(M,g)$ on a neighborhood of $p$ so that it is self-dual.
\subsection{Non-Isotropic Case}
First, we consider non-isotropic half conformally flat gradient $f$-almost Ricci solitons. That is, $||\nabla f||\neq 0$ so the level sets of f are non-degenerate hypersurfaces. Here, we also assume that $\nabla f$ is of constant length. 
Since Cotton tensor is a constant multiple of divergence of Weyl tensor, by virtue of $\eqref{lemma 2}$ and $\eqref{lemma brozos 1}$, we can express the self-duality condition as follows:
\begin{align}
\label{half 1}
3&[Q(e_i)g(e_1,\nabla f)-Q(e_1)g(e_i,\nabla f)]+Ric(e_1,\nabla f)e_i-Ric(e_i,\nabla f)e_1\\ \notag
+&r[g(e_i,\nabla f)e_1-g(e_1,\nabla f)e_i]-\frac{6}{f}[g(\nabla f, e_1)\nabla_{e_i}\nabla f-g(\nabla f, e_i)\nabla_{e_1}\nabla f]\\ \notag
+&2\frac{\Delta f}{f}[g(\nabla f, e_1){e_i}-g(\nabla f, e_i)e_1]-\frac{1}{f}[e_1(||\nabla f||^2)e_i-e_i(||\nabla f||^2)e_1]\\ \notag 
=&\sigma_{ijk}\varepsilon_j \varepsilon_k\Big[3[Q(e_k)g(e_j,\nabla f)-Q(e_j)g(e_k,\nabla f)]+Ric(e_j,\nabla f)e_k-Ric(e_k,\nabla f)e_j\\ \notag 
+&r[g(e_k,\nabla f)e_j-g(e_j,\nabla f)e_k]-\frac{6}{f}[g(\nabla f, e_j)\nabla_{e_k}\nabla f-g(\nabla f, e_k)\nabla_{e_j}\nabla f]\\ \notag
+&2\frac{\Delta f}{f}[g(\nabla f, e_j){e_k}-g(\nabla f, e_k)e_j]-\frac{1}{f}[e_j(||\nabla f||^2)e_k-e_k(||\nabla f||^2)e_j]\Big]
\end{align}
for $i,j,k\{2,3,4\}$, where $\varepsilon_i=<e_i,e_i>$. Since $||\nabla f||\neq 0$, we can normalize $\nabla f$ to be a unit and complete it to an orthonormal frame $\{E_i:i=1,2,3,4\}$, where $E_1=\frac{\nabla f}{||\nabla f||}$. Then, normalizing $\eqref{half 1}$ with respect to this orthonormal frame, we obtain 
\begin{align}
\label{half 2}
3&Ric(E_i,Z)g(E_1,\nabla f)+Ric(E_1,\nabla f)g(E_i,Z)-Ric(E_i,\nabla f)g(E_1,Z)\\ \notag
-&rg(E_1,\nabla f)g(E_i,Z)-\frac{6}{f}g(\nabla f, E_1)g(\nabla_{E_i}\nabla f, Z) \\ \notag
+&2\frac{\Delta f}{f}g(\nabla f, E_1)g(E_i,Z)-\frac{1}{f}[E_1(||\nabla f||^2)g(E_i,Z)-E_i(||\nabla f||^2)g(E_1,Z)]\\ \notag
=&\sigma_{ijk}\varepsilon_j \varepsilon_k\Big[
Ric(E_j,\nabla f)g(E_k,Z)-Ric(E_k,\nabla f)g(E_j,Z) \\ \notag
-&\frac{1}{f}[E_j(||\nabla f||^2)g(E_k,Z)-E_k(||\nabla f||^2)g(E_j,Z)]\Big].
\end{align}
Now, first putting $Z=E_1$ in $\eqref{half 2}$ and using the fact that  $||\nabla f||$ is non-zero constant, we get 
\begin{equation}
\label{half 3}
Ric(E_i,E_1)=0, \ \ \forall \ i\in \{2,3,4\}
\end{equation}
which means that $\nabla f$ is an eigenvector of the Ricci operator. 

Next, putting $Z=E_j$ in $\eqref{half 2}$, we get
\begin{equation}
\label{half 4}
Ric(E_i,E_j)=0, \ \ \forall \ i\neq j.
\end{equation}
That is, the Ricci operator can be diagonalizable  on the basis  $\{E_1,E_2,E_3,E_4\}$.

Finally, setting $Z=E_i$ in $\eqref{half 2}$, we  obtain
\begin{align}
\label{half 5}
\varepsilon_1Ric(E_1,E_1)-r+ 3\varepsilon_iRic(E_i,E_i)-\frac{6}{f}Hessf(E_i,E_i)
-2\frac{\Delta f}{f}-\frac{1}{f}E_1(||\nabla f||)=0.
\end{align}
In view of $\eqref{half 5}$, the fundamental equation of gradient $f$-almost Ricci soliton yields
\begin{equation}
\label{half 6}
Hessf(E_i,E_i)=- \frac{1}{f}Ric(E_i,E_i)+\frac{\lambda}{f},  \ \ \forall \ i\in \{2,3,4\}.
\end{equation}
As a consequence of the last two equations, we obtain
\begin{equation}
\label{half 6.}
Hessf(E_i,E_i)=\gamma g(E_i,E_i)
\end{equation}
where $\gamma =\frac{f}{3f^2+6}[Ric(E_1,E_1)+3\lambda -r-2\frac{\Delta f}{f}]$. On the other hand, when $i\neq j$, $Ric(E_iE_j)=\frac{2}{f}Hessf(E_i,E_j)=0$.
Thus, the level hypersurface of $f$ are totally umbilical. Moreover, for $i\in \{2,3,4\}$,
$Hessf(E_i,E_1)=g(hes_f(E_1),E_i)=0$ and so $hes_f(E_1)=0$. Thus, the 1-dimensional distribution $Span\{E_1\}$ is totally geodesic. Then, by using the reference \cite{ponge-twisted}, we conclude that $(M,g)$ can be locally decomposed as a twisted product $I\times_{\varphi}F$. Also,  as the Ricci operator is diagonalizable,
the twisted product reduces to a warped product (see \cite{lopez-rio-twisted}). That is, $(M,g)$ is a 4-dimensional self-dual warped product manifold. Thus, it becomes also anti-self-dual. Therefore, it is necessarily locally conformally flat and as a consequence of this result, the fiber $F$ becomes an Einstein manifold. Hence we can state the following theorem:

\begin{theorem}\label{thm:1}
Let $(M, g, f )$ be a  four-dimensional half conformally flat gradient $f$-almost Ricci soliton of neutral signature. Then, if $||\nabla f||=$constant $\neq 0$,  then $(M, g)$ is locally isometric to a {warped product} of the form $I \times_{\varphi} N$, where $I \subset \mathbb{R}$ and $N$ is an Einstein manifold. Furthermore, $(M, g)$ is locally conformally flat.

\end{theorem}

\subsection{Isotropic Case}

In this case,  the level hypersurfaces of the potential function are now degenerate. That is, in contrast to the non-isotropic case, we have $||\nabla f||=0$. But, as $\nabla f \neq 0$, we can complete it to a local pseudo-orthonormal frame $\mathcal{B}=\{\nabla f, U,V, W\}$, i.e., the only non-zero components of the metric $g$ are
\begin{equation}
\label{half 7}
g(\nabla f, V)= g(U,W)=1.
\end{equation}  
Since $||\nabla f||=g(\nabla f, \nabla f)=0$, taking covariant derivative, we get for all $X\in TM$, $Hessf(\nabla f, X)=g(hes_f(\nabla f), X)=0$,  which yields $hes_f(\nabla f)=0$. Combining this with the fundamental equation of gradient $f$-almost Ricci soliton, we have 
\begin{equation}
\label{half 8}
Q(\nabla f)=\lambda \nabla f
\end{equation}
where $Q$ is the Ricci operator. Thus, $\nabla f$ is an eigenvector of the Ricci operator corresponding to the eigenvalue $\lambda$. 

Now, since $(M,g,f)$ is half conformally flat, without lose of generality we can assume that it is self-dual. Then by the equations   $\eqref{lemma:brozos}$ and $\eqref{lemma brozos 2}$, for the pseudo-orthonormal frame  $\mathcal{B}=\{\nabla f, U,V, W\}$, we have
\begin{equation}
\label{half 9}
W(\nabla f,V,X,Y)=W(U,W,X,Y), \ \ W(\nabla f,W,X,Y)=0, \ \ W(U,V,X,Y)=0,
\end{equation}
for all $ X, Y$. Setting $Y=\nabla f$ in the first equation of $\eqref{half 9}$ and using the equations $\eqref{lemma 2}$ and $\eqref{half 8}$, we get
\begin{equation}
\label{half 10}
W(\nabla f,V,X,\nabla f)=\Big(\frac{r-4\lambda}{6}+\frac{\Delta f}{3f} \Big) g(\nabla f,X).
\end{equation}
Putting $Y=\nabla f$ in the second equation of $\eqref{half 9}$ and using the equations $\eqref{lemma 2}$ and $\eqref{half 8}$,  we get $W(U,W,X,\nabla f)=0$. Combining the last relation with $\eqref{half 10}$, we have either  $r=4\lambda -2\frac{\Delta f}{f}$ or for all $X$, $g(\nabla f,X)=0$. But the last one gives rise to $\nabla f=0$, which is a contradiction. Thus, we get  $r=4\lambda -2\frac{\Delta f}{f}$. 

Similarly, putting $Y=\nabla f$ in the third equation of $\eqref{half 9}$ and using the equations $\eqref{lemma 2}$, $\eqref{half 8}$ and the fundamental equation of the gradient $f$-almost Ricci soliton, we get
\begin{equation}
\label{half 11}
Ric(U,X)=\lambda g(U,X), \ \ \forall \ X \ \in \ TM
\end{equation}
which shows that $U$ is an eigenvector of the Ricci operator corresponding to the eigenvalue $\lambda$. 

Finally, putting $X=V$ in the second equation of $\eqref{half 9}$ and using the equations $\eqref{lemma 2}$ and $\eqref{half 8}$, we get
\begin{align}
\label{half 12}
\frac{\lambda}{6}&g(Y,W)-\frac{r}{6}g(Y,W)- \frac{1}{2}Ric(V,W)g(Y,\nabla f)-\frac{1}{f}Hessf(Y,W)\\ \notag
+&\frac{1}{f}g(Y,\nabla f)Hessf(V,W)+\frac{\Delta f}{3f}g(Y,W)+Ric(Y,W)=0, \ \ \forall \ Y \ \in \ TM.
\end{align}
Now, putting $Y=W$ and $Y=\nabla f$ in $\eqref{half 12}$ and using the fundamental equation, respectively we obtain
\begin{equation}
\label{half 13}
Ric(W,W)=Ric(\nabla f, W)=0.
\end{equation}
As a consequence of these obtained results, the Ricci operator can be written as in the following form:
\begin{displaymath}
\mathbf{Q} =
\left( \begin{array}{cccc}
\lambda & 0 & \alpha & \beta \\
0 & \lambda & \beta & 0 \\
0 & 0 & \lambda & 0\\
0 & 0 & 0 & \lambda\\
\end{array} \right)
\end{displaymath}
for some non-zero functions $\alpha$ and $\beta$ on $M$.
Now,  we set $\mathcal{D}=\{\nabla f, U\}$, which is 2-dimensional null-distribution. Since  $\mathcal{B}=\{\nabla f, U,V, W\}$ is pseudo-orthonormal basis in which the only non-zero components of the metric $g(\nabla f, V)= g(U,W)=1$, we have
\begin{equation}
\label{half 14}
g(\nabla_X\nabla f, \nabla f)=g(\nabla_XU,U)=0.
\end{equation}
Also, by using the fundamental equation of gradient $f$-almost Ricci soliton and $\eqref{half 11}$, we have $hes_f(U)=0$. Thus, we obtain
\begin{equation}
\label{half 15}
g(\nabla_X\nabla f, U)=g(\nabla_XU,\nabla f)=0.
\end{equation}
Therefore, $\nabla \mathcal{D} \subset \mathcal{D}$, i.e., the distribution  $\mathcal{D}$ is invariant under the Levi-Civita connection $\nabla$. Thus, $\mathcal{D}$ is a null-parallel distribution and so $(M,g,f,\lambda)$ is locally a Walker manifold. Hence, we can state that:

\begin{theorem}\label{thm:2}
Let $(M, g, f )$ be a  four-dimensional half conformally flat gradient $f$-almost Ricci
soliton of neutral signature. Then if $||\nabla f|| = 0$,  then $(M,g)$ is locally a {Walker manifold}, where $r=4\lambda -2 \frac{\Delta f}{f}$.
\end{theorem}

\begin{remark}
From \cite{guler:3}, if $dimM=4$ and $dimD=2$, the metric $g$  of the Walker manifold $(M,g,D)$ has the neutral signature $(-,-,+,+)$ and in suitable coordinates it has the form
\begin{displaymath}
\mathbf{g(x,y,z,t)} =
\left( \begin{array}{cccc}
0 & 0 & 1 & 0 \\
0 & 0 & 0 & 1 \\
1 & 0 & a(x,y,z,t) & c(x,y,z,t)\\
0 & 1 & c(x,y,z,t) & b(x,y,z,t)\\
\end{array} \right)
\end{displaymath}
for some function $a(x,y,z,t)$, $b(x,y,z,t)$ and $c(x,y,z,t)$, where $D=<\frac{\partial}{\partial x}, \frac{\partial}{\partial y}>$. 
\end{remark}

Some geometric properties of four-dimensional Walker metrics satisfying $c=0$ and $c=$constant were investiged in \cite{batat} and \cite{azimpour}, respectively.  Here we are intereseted in the particular case by choosing $a=c=0$ in the metric $g(x,y,z,t)$ and we  construct the following:

\begin{example}
Let $\{t,x,y,z\}$ be the local coordinates  with respect to the local frame fields $\{\partial_t,\partial_x,\partial_y,\partial_z\}$  and we consider the metric given by 
\begin{equation}
\label{gb metric}
(g_b)=ds^2=2dtdy+2dxdz+bdz^2
\end{equation} 
for some smooth function $b$. Then by long and straightforward calculations, the non-zero components of the Riemannian curvature tensor Ricci tensor are as follows:
\begin{equation}
\label{4 dim Walker riemann curvature}
R(\partial_y,\partial_t)\partial_y=\frac{1}{2}b_{yy}\partial_y, \ \ \ \ \ \ R(\partial_y,\partial_t)\partial_t=\frac{1}{2}bb_{yy}\partial_y-\frac{1}{2}b_{yy}\partial_t
\end{equation}
\begin{align}
\label{4 dim Walker ricci curvature}
Ric(\partial_x,\partial_t)=\frac{1}{2}b_{xy}, \  
Ric(\partial_y,\partial_t)=&\frac{1}{2}b_{yy}, \\ \notag
Ric(\partial_z,\partial_t)=\frac{1}{2}b_{yz}, \ \
 Ric(\partial_t,\partial_t)=&-b_{zx}+\frac{1}{2}bb_{yy}
\end{align}
To construct a gradient $f$-almost Ricci soliton structure on $(M,g_f)$, by using $\eqref{gradient h-almost}$, $\eqref{gb metric}$ and $\eqref{4 dim Walker ricci curvature}$, we need to find the potential function $f$ and soliton function $\lambda$ satisfying following system of differential equations:
\begin{equation}
\label{system.1}
 \left\{ \begin{array}{ll}
\partial_xf_x=0, \ \ \ \ \ \partial_yf_x=0, 
f\partial_zf_x=\lambda, \ \ \
\partial_tf_x-\frac{1}{2}b_xf_y=0, \ \\
\\
\partial_yf_y=0, \ \ \ \ \ \partial_zf_y=0,\\
\\
\frac{1}{2}b_{yy}+f[\partial_tf_y-\frac{1}{2}b_yf_y]=\lambda, \ \\
\\
\partial_zf_z=0, \ \ \
\frac{1}{2}b_{yz}+f[\partial_tf_z-\frac{1}{2}b_zf_y]=0, \\
\\
-b_{xz}+\frac{1}{2}bb_{yy}+f[\partial_tf_t+\frac{1}{2}(b_zf_x-b_tf_y-bb_{yy}f_y+b_xf_z+b_yf_t)]=\lambda b,
\end{array} \right.
\end{equation} 
where all indices denotes the partial derivatives with respect to corresponding coordinates.

First, without lose of generality, we assume that $f$ depends only on $y$.  Then, by the system $\eqref{system.1}$, we obtain  $\lambda =0$, $
f=\alpha y +\beta$;  \ for \ any \ $\alpha,\beta \in \mathbb{R}$, \  and $b$ is a differentiable function of only $y$ satisfying the resulting system of $\eqref{system.1}$.
Hence $(M,g_b, \nabla f,0)$ is the {steady gradient $f$-almost Ricci soliton}, where $f$ and $b$ defined as above.

Secondly, to construct another explicit example we may consider $f$ as a seperable function of $x$ and $z$. In this case, we assume that  $f(x,z)=xz$ and by the system $\eqref{system.1}$, we obtain $\lambda =xz$ and $b=xy^2z+c$, for \ any \ $c \in \mathbb{R}$, which yields  the {non-steady gradient $f$-almost Ricci soliton}  $(M, g_b, \nabla f, \lambda )$.

\end{example}

\section{More Physical Results: Standard Static Spacetimes}

First, to fix the notation, we may give a formal definition of standard static spacetime:
\begin{definition}
Let $(F, \bar{g} )$ be an $n$-dimensional Riemannian manifold and
$h : F \rightarrow (0,\infty)$ be a smooth function. Then $(n+1)$-dimensional product
manifold $M=I\times  F$ furnished with the metric tensor $\tilde{g}= -h^2dt^2 \oplus \bar{g}$ is called
a standard static space-time (briefly $SSS$-time) and it is denoted by $M=I_h \times F$, where $dt^2$ is the
Euclidean metric tensor on $I$. 
\end{definition}
Here throughout this section each object denoted by "tilde" is assumed to be from the $SSS$-time and
each object denoted by "bar" is assumed to be from the fiber $F$.
This is just the warped product $F\times_fI $ with time coordinate written first. Thus, by contrast with a Robertson-Walker spacetime, space remains the same but time is warped, \cite{oneill}. These type of metrics play an important role in the search of solutions of the Einstein field equations so they have been previously studied by many authors.
Two of the most famous examples of standard static space-times are the
Minkowski space-time and the Einstein's static universe \cite{5.b}.
% which is $\mathbb{R} \times \mathbb{S}^3$ equipped with the metric  
%\begin{equation*}
%\label{example 1}
%$g=-dt^2+(dr^2+sin^2rd\theta^2+sin^2rsin^2\theta d\phi^2)$,
%\end{equation*}
%where $\mathbb{S}^3$ is the usual 3-dimensional Euclidean sphere and the warping function
%$f = 1$.
Another well-known examples are  the universal covering space of anti-de Sitter space-time and the Exterior
Schwarzschild space-time (for more details, \cite{13.b}).

In the late of 1980's, Besse \cite{besse} posed an important question about the construction of  Einstein manifolds that have warped product stuctures. More recently, some necessary and sufficient conditions are obtained to construct Einstein warped product manifold. If it is more clearly stated, $(M\times N,{g})$ endowed with the metric ${g}=g_M\oplus e^{\frac{-2f}{m}}{g_N} $ is Einstein if and only if the fiber manifold is Einstein and the $m$-Bakery-Emery Ricci tensor of the base manifold is proportional to its metric, i.e., it is quasi Einstein  \cite{kims}.
Thus, it would be a very interesting idea to extend these studies  to  the generalizations of Ricci solitons  having the standard static spacetime metric structure. This will be the point of view of this section. 

The next lemma includes some geometrical objects such as the Levi-Civita connection and the Ricci tensor  of the standard static spacetime metric:

 \begin{lemma}\cite{oneill}\label{lemma:Ricci curvature of SSS} 
Under the above notations, on the $h$-associated $SSS$-time, for any vector fields $X,Y\in \chi(I)$ and $U,V\in \chi(F)$, one has:
\begin{itemize} 
\item[(i)] the only non-vanishing components of the Levi-Civita connection of $\tilde{\nabla}$ of $M$ are given by:
\begin{itemize}
\item[(1)]$\tilde{\nabla}_XY=\nabla_XY-\frac{\tilde{g}(X,Y)}{h}\nabla h$, \ \
\item[(2)] $\tilde{\nabla}_XV=\tilde{\nabla}_VX=(\frac{Vh}{h})X$, \ \  (3) $\tilde{\nabla}_UV=\bar{\nabla}_UV$, 
\end{itemize}
\item[(ii)] the only non-vanishing components of the  Ricci tensor of $\tilde{Ric}$ of $M$ are given by:
\begin{itemize}
\item[(1)] $\tilde{Ric}(X,Y)=-h\Delta hg(X,Y)$, \   
\item[(2)] $\tilde{Ric}(U,V)=\bar{S}(U,V)-\frac{Hess h(U,V)}{h}$;
\end{itemize}
where  $Hess h$ and $\Delta h$ denotes the Hessian and laplacian of $h$ on $F$, respectively.
\end{itemize} 
Moreover, the scalar curvatures of  $I_h\times F$ is given by $\tilde{r}=\bar{r}-2\frac{\Delta h}{h}$.
\end{lemma}

It is known that a smooth manifold $(M^n,g)$ $(n>2)$  is said to be  generalized quasi Einstein manifold \cite{catino} if there exist three smooth functions $f, \alpha$ and $\lambda$ such that
\begin{equation}
\label{generalized quasi Einstein}
Ric+Hessf-\alpha df \otimes df=\lambda g,
\end{equation} 
which is simply denoted by $(M^n,g,\nabla f, \lambda)$.
If $\alpha =\frac{1}{m}$, for positive integer $0<m<\infty$ then $M$ is called an $m$-generalized quasi Einstein manifold. If in addition, $\lambda \in \mathbb{R}$, then $M$ reduces to an $m$-quasi Einstein manifold. 
Moreover, defining a function $\varphi=e^{\frac{-f}{m}}$, we get  $\frac{m}{\varphi}Hess\varphi=-Hessf+\frac{1}{m}df \otimes df$. By using this relation in the fundamental equation, we obtain 
\begin{equation}
\label{3.3.3.}
{Ric}-\frac{m}{\varphi} Hess\varphi=\lambda {g}.
\end{equation}
Thus, all $m$-generalized quasi Einstein manifold is a gradient $(\frac{-m}{\varphi})$-almost Ricci soliton.
After this transformation, we can prove the following:

\begin{theorem}\label{thm:3.3.1}
Let $(M=I_h\times F, \tilde{g}, f,\lambda) $ be an  $m$-generalized quasi Einstein standard static spacetime. Then, the potential function  $f$ depends only on the fiber $(F,\bar{g})$ which satisfies the Ricci-Hessian class type equation 
\begin{equation}
\label{3.3.10}
\bar{Ric}+Hess\psi-d\psi \otimes d\psi=\lambda \bar{g}+\frac{m}{\varphi}Hess\varphi .
\end{equation}
%containing the class of quasi Einstein metrics. 
\end{theorem}

\begin{proof}
Since $(M=I_h\times F, \tilde{g}, f,\lambda) $ is an $m$-generalized quasi Einstein $SSS$-time, by $\eqref{3.3.3.}$ its Ricci tensor satisfies 
\begin{equation}
\label{3.3.3}
\tilde{Ric}-\frac{m}{\varphi} \tilde{Hess}\varphi=\lambda \tilde{g}.
\end{equation}
In view of the equation $\eqref{3.3.3}$, Lemma $\ref{lemma:Ricci curvature of SSS}$-(ii) and the definition of the $SSS$-time metric, for all $X\in \chi(I)$ and $V\in \chi(F)$ we have $ (\tilde{Hess}\varphi)(X,V)=0$. Then, again by using Lemma $\ref{lemma:Ricci curvature of SSS}$-(i) we get the following separation for the tangential and the normal parts of the vectors
\begin{align}
\label{3.3.4}
\tilde{Hess}( \varphi)(X,V)
=&\tilde{g}(\tilde{\nabla}_X tan(grad_{\tilde{g}}\varphi),V)+\tilde{g}(\tilde{\nabla}_X nor(grad_{\tilde{g}}\varphi),V)\\ \notag
=&hdh(V)g(X,tan(grad_{\tilde{g}}\varphi))=0.
\end{align}
From the last row, since $h>0$, we have either $dh(V)=0$, for any $V\in \chi(F)$ or $g(X,tan(grad_{\tilde{g}}\varphi))=0$.  
But the first case implies $h$ is a constant, which makes the  $SSS$-time trivial. Thus, we have $g(X,tan(grad_{\tilde{g}}\varphi))=0$, for all $X$. Therefore, the function $\varphi$ and so the potential function $f$ depends only on fiber.   
That is, $nor(grad_{\tilde{g}}\varphi)=grad_{\bar{g}}\varphi$. Since $g=-(dt)^2$, we may choose $X=Y=\partial_t$. Then, by $\eqref{3.3.3}$, we obtain
\begin{equation}
\label{3.3.5}
\tilde{Ric}(\partial_t,\partial_t)-\frac{m}{\varphi} \tilde{Hess}\varphi(\partial_t,\partial_t)=\lambda \tilde{g}(\partial_t,\partial_t),
\end{equation}
from which  following relation holds:
\begin{equation}
\label{3.3.6}
\Delta h+\frac{m}{\varphi}(grad_{\bar{g}}\varphi)(h)=-h\lambda .
\end{equation}
Similarly, for all $V,W\in \chi(F)$, we get
\begin{equation}
\label{3.3.7}
\tilde{Ric}(V,W)-\frac{m}{\varphi} \tilde{Hess}\varphi(V,W)=\lambda \tilde{g}(V,W).
\end{equation}
Again, by the definition of the metric $\tilde{g}$ and the  Lemma $\ref{lemma:Ricci curvature of SSS}$-(ii), $\eqref{3.3.7}$ yields
\begin{equation}
\label{3.3.8}
\bar{Ric}(V,W)-\frac{1}{h}Hessh(V,W)-\frac{m}{\varphi} Hess\varphi(V,W)=\lambda \bar{g}(V,W).
\end{equation}
Note that the following relation holds for any smooth function $h$;
\begin{equation}
\label{3.3.9}
\nabla^2(lnh)=\frac{1}{h}\nabla^2h-\frac{1}{h^2}dh\otimes dh.
\end{equation}
If we take $\psi=-lnh$, then by virtue of $\eqref{3.3.9}$, $\eqref{3.3.8}$ yields the equation $\eqref{3.3.10}$.
 Hence the proof is completed.
\end{proof}

Then, we shall actually prove a slightly stronger theorem, which is the converse case of the above theorem. This also allows us to obtain sufficient conditions about in which conditions we can get standard static metric structure:
\begin{theorem}\label{thm:3.3.2}
Let $(F^n,\bar{g}, f,-\frac{2\Delta h+h \Delta \varphi}{h})$ be a non-steady $m$-generalized quasi Einstein manifold. Then the standard static spacetime $(I_h\times F, \tilde{g}, \nabla \varphi, -\frac{2\Delta h+h \Delta \varphi}{h})$ is a gradient almost Ricci soliton, where  $\varphi=\frac{f}{2}$ and $h=e^{-\varphi}$.
\end{theorem}

\begin{proof}
We assume that $(F^n,\bar{g}, f,\lambda)$ be a non-steady $m$-generalized quasi Einstein metric. Thus, we have
\begin{equation}
\label{3.3.11}
\bar{Ric}+Hessf-\frac{1}{m}df \otimes df=\lambda \bar{g}; \ \ \ \lambda \in C^{\infty}(F).
\end{equation}
Define $\varphi=\frac{f}{2}$, $h=e^{\frac{-\varphi}{\alpha}}$ and let $m=4\alpha$. Then, we get
\begin{equation}
\label{3.3.12}
\frac{\alpha}{h}\nabla^2h=\frac{1}{\alpha}d\varphi \otimes d\varphi-\nabla^2\varphi .
\end{equation}
Then, in view of $\eqref{3.3.12}$, $\eqref{3.3.11}$ can be written as follows (when $\alpha=1$);
\begin{equation}
\label{3.3.13}
\bar{Ric}+\nabla^2\varphi=\lambda \bar{g}+\frac{1}{h}\nabla^2h.
\end{equation}
By virtue of $\eqref{3.3.13}$ and the Lemma $\ref{lemma:Ricci curvature of SSS}$,  for all $V,W\in \chi(F)$ the relation
\begin{equation}
\label{3.3.14}
\tilde{Ric}(V,W)+\nabla^2\varphi(V,W)=\lambda\tilde{g}(V,W)
\end{equation}
holds. Since for all $X\in \chi(I)$ and $V\in \chi(F)$, $\tilde{g}(X,V)=0$
 and $\tilde{Ric}(X,V)=0$, the relation 
\begin{equation}
\label{3.3.15}
\tilde{Ric}(X,V)+\nabla^2\varphi(X,V)=\lambda\tilde{g}(X,V)
\end{equation}
trivially holds. Also, from the Lemma $\ref{lemma:Ricci curvature of SSS}$ and the Theorem $\ref{thm:3.3.1}$,  for all $X,Y\in \chi(M)$, i.e., $X=Y=\partial_t$,
\begin{equation}
\label{3.3.16}
(\nabla^2\varphi)(\partial_t,\partial_t)=\tilde{g}(\tilde{\nabla}_{\partial_t}\nabla\varphi,\partial_t)=-h\nabla\varphi(h).
\end{equation}
As $\Delta h=div(\nabla h)$, from assumptions, we have $\Delta h+\nabla\varphi(h)=-h\Delta \varphi$.
Thus, if we choose the soliton constant $\lambda =-\frac{2\Delta h+h \Delta \varphi}{h}$, the relation 
\begin{equation}
\label{3.3.17}
\tilde{Ric}(\partial_t,\partial_t)+\nabla^2\varphi(\partial_t,\partial_t)=\lambda\tilde{g}(\partial_t,\partial_t)
\end{equation}
also holds. Therefore, in each case $(I_h\times F, \tilde{g}, \nabla \varphi, -\frac{2\Delta h+h \Delta \varphi}{h})$ satisfies the fundamental equation of a gradient almost Ricci soliton, which completes the proof.
\end{proof}

On the other hand, if we assume that   $(I_h\times F^n,\tilde{g}, f, \lambda)$ is an  $m$-generalized quasi Einstein $SSS$-time,  by virtue of $\eqref{3.3.12}$, we have
\begin{equation}
\label{3.3.18}
\tilde{Ric}+\nabla^2\varphi=\lambda \tilde{g}+\frac{\alpha}{h}\nabla^2h.
\end{equation}
Additionally, we assume that $(F^n,\bar{g}, f, \lambda)$ satisfies the relation
\begin{equation}
\label{3.3.19}
-2\lambda d\varphi+d[(2-n-\alpha)\lambda+||\nabla \varphi||^2-\Delta \varphi-\frac{\alpha}{h}\nabla\varphi(h)]=0.
\end{equation}
Contracting $\eqref{3.3.18}$, the scalar curvature is obtained as follows:
\begin{equation}
\label{3.3.20}
\tilde{r}=n\lambda+\frac{\alpha}{h}\Delta h-\Delta \varphi
\end{equation}
and so its differential is
\begin{equation}
\label{3.3.21}
d\tilde{r}=nd\lambda+\frac{\alpha}{h}d(\Delta h)-\frac{\alpha}{h^2}dh\Delta h-d(\Delta \varphi).
\end{equation}
Notice that  for all smooth  function $\varphi$ and $(0,2)$-tensor $T$, the following general facts are well known:
\begin{equation}
\label{3.3.22}
div(\varphi T)=\varphi divT+T(\nabla\varphi,\cdot).
\end{equation}
\begin{equation}
\label{3.3.23}
div(\nabla^2 \varphi)=Ric(\nabla\varphi,\cdot)+(\Delta\varphi) \ \ \ \textrm{and} \ \  \ \frac{1}{2}d(||\nabla\varphi||^2)=\nabla^2\varphi(\nabla\varphi,\cdot).
\end{equation}
Taking divergence of $\eqref{3.3.18}$ and using the equations $\eqref{3.3.19}$, $\eqref{3.3.22}$, $\eqref{3.3.23}$ and the contracted second Bianchi Identity, we get $d[\lambda h^2+h\Delta h+(\alpha-1)||\nabla h||^2-h\nabla\varphi(h)]=0$ and so  for some constant $c_0$, we have
\begin{equation}
\label{3.3.24}
 \lambda h^2+h\Delta h+(\alpha-1)||\nabla h||^2-h\nabla\varphi(h)=c_0.
\end{equation}
Now, we consider the elliptic operator of second order given by
\begin{equation}
\label{3.3.25}
\varepsilon(\cdot)=\Delta(\cdot)-\nabla\varphi(\cdot)+\frac{\alpha-1}{h}\nabla h(\cdot).
\end{equation}
Then by virtue of $\eqref{3.3.24}$, $\eqref{3.3.25}$ yields
\begin{equation}
\label{3.3.26}
\Delta(h)-\nabla\varphi(h)+\frac{\alpha-1}{h}||\nabla h||^2=\frac{c_0-\lambda h^2}{h}.
\end{equation} 
If we assume that $h$ is bounded function which has maximum and minimum values at some points $p$ and  $q$ on $F^n$, respectively, then we have
\begin{equation}
\label{3.3.27}
\nabla h(p)=0=\nabla h(q) \ \ \ \textrm{and} \ \ \ \Delta h(p)\leqslant 0 \leqslant \Delta h(q).
\end{equation}
If we take $\lambda \leqslant 0$ such that $\lambda (p)\leqslant \lambda (q)$, then by $\eqref{3.3.24}$ and $\eqref{3.3.27}$, we get
\begin{equation}
\label{3.3.28}
c_0-\lambda (p)h^2(p)=c_0-\lambda (q)h^2(q).
\end{equation}
Now, if $\lambda(q)\neq 0$, then by $\eqref{3.3.28}$, we have $h^2(p)\leqslant h^2(q)$ which implies for all $p,q\in F$, $h(p)=h(q)$. Therefore $h$ is constant and so is $\lambda$. 
If $\lambda(p)=0$, then by $\eqref{3.3.24}$, $c_0=0$. In this case, from $\eqref{3.3.24}$ and $\eqref{3.3.25}$ we obtain  $\varepsilon(h)\geqslant 0$. 
It is known that the Strong Maximum Principle tells us that for a solution of an elliptic
equation, extrema can be attained in the interior if and only if the function is a constant. Hence we conclude that $h$ is constant and so $\lambda =0$. As $h$ is constant, the potential function $f$ of $(M=I_h\times F^n, \tilde{g},f, \lambda)$ is also constant. Thus, $M$ is trivial. 

Hence together with this result, we have obtained  an answer  to the question posed by Besse in \cite{besse}, for $m$-generalized quasi Einstein manifolds with $\lambda \leqslant 0$, which also extends the results obtained for the warped product manifolds in \cite{2015pina} to the standard static spacetimes. Therefore, we can state the following:

\begin{corollary}\label{thm:3.3}
Let $(M=I_h\times F^n, \tilde{g},f, \lambda)$ be an  $m$-generalized quasi Einstein standard static spacetime satisfying $\eqref{3.3.19}$ with bounded warping function $h$ and  $\lambda \leqslant 0$ such that $\lambda(p)\leqslant \lambda (q)$, where $p$ and $q$ are maximum and minimum points of $h$, respectively. Then, $(M=I_h\times F^n, \tilde{g},f, \lambda)$  is trivial  Lorentzian product manifold.  
\end{corollary}

\begin{acknowledgements}
The author  would like to thank  anonymous referees for all their useful remarks and comments. Also, the  author is supported by The Scientific and Technological Research Council of Turkey (T\"{U}B\.{I}TAK) BIDEB-2218  postdoctoral programme (Grant Number: 1929B011800249).

\end{acknowledgements}

% Authors must disclose all relationships or interests that 
% could have direct or potential influence or impart bias on 
% the work: 
%
% \section*{Conflict of interest}
%
% The authors declare that they have no conflict of interest.

% BibTeX users please use one of
%\bibliographystyle{spbasic}      % basic style, author-year citations
%\bibliographystyle{spmpsci}      % mathematics and physical sciences
%\bibliographystyle{spphys}       % APS-like style for physics
%\bibliography{}   % name your BibTeX data base

% Non-BibTeX users please use

\end{document}